\documentclass[11pt,letterpaper]{amsart}

\usepackage{amsmath}
\usepackage{graphicx}
\usepackage{amsmath,amssymb,amsfonts,amsthm,graphicx,amsopn,mathrsfs}
\usepackage[shortlabels]{enumitem}
\usepackage{tikz-cd}
\usepackage{tikz}
\usepackage{stmaryrd,url}
\usepackage{mathdots, verbatim,mathtools}
\usepackage{array,float}
\usepackage{fullpage}
\usepackage[utf8]{inputenc}
\usepackage{multirow}
\usepackage[normalem]{ulem}
\usepackage{hyperref}
\usepackage{cleveref}
\usepackage{algorithmic,algorithm}
\usepackage{caption}

\tikzstyle{nodal}=[circle,draw,fill=black,inner sep=0pt, minimum width=4pt]
\tikzstyle{half-fiber}=[rectangle,draw=black,thick,inner sep=0pt, minimum width=5pt, minimum height=5pt]
\tikzset{double distance = 2pt}

\DeclareMathOperator{\rk}{rk}
\DeclareMathOperator{\Aut}{Aut}
\DeclareMathOperator{\id}{id}

\DeclareMathOperator{\disc}{disc}

\DeclareMathOperator{\Pic}{Pic}

\DeclareMathOperator{\Or}{O}

\usepackage[sorting=nyc,doi=false,isbn=false,url=false,giveninits=true,maxbibnames=99]{biblatex}

\DeclareSortingTemplate{nyc}{
  \sort{
    \field{presort}
  }
  \sort[final]{
    \field{sortkey}
  }
  \sort{
    \field{sortname}
    \field{author}
    \field{editor}
    \field{translator}
    \field{sorttitle}
    \field{title}
  }
  \sort{
    \field{sortyear}
    \field{year}
  }
  \sort{\citeorder}
}

\AtEveryBibitem{%
\ifentrytype{article}{
    \clearfield{number}
}{}
}
\renewbibmacro{in:}{%
  \ifentrytype{article}{}{\printtext{\bibstring{in}\intitlepunct}}}
\AtEveryBibitem{%
\ifentrytype{book}{
    \clearfield{pages}
}{}
}

\title{On hyperbolic lattices whose reduced automorphism group is virtually abelian}
\author{Simon Brandhorst}
\address{Simon Brandhorst:
Fakult\"at f\"ur Mathematik und Informatik, Universit\"at des Saarlandes, Campus E2.4, 66123 Saarbr\"ucken, Germany}
\email{brandhorst@math.uni-sb.de}

\author{Giacomo Mezzedimi}
\address{Giacomo Mezzedimi:
Mathematisches Institut, Universit\"at Bonn, Endenicher Allee 60, 53115 Bonn, Germany.}
\email{mezzedim@math.uni-bonn.de}

\thanks{This work is supported by SFB
195 No. 286237555 of DFG}

\date{\today}

\newcommand{\QQ}{\mathbb{Q}}
\newcommand{\RR}{\mathbb{R}}
\newcommand{\ZZ}{\mathbb{Z}}

\newcommand{\HH}{\mathbb{H}}

\newcommand{\AutredL}{\Aut_{\rm red}(L)}
\newcommand{\AutredLpr}{\Aut_{\rm red}(L')}
\newcommand{\AutredM}{\Aut_{\rm red}(M)}

\newcommand{\D}{\mathcal{D}}

\newtheorem{theorem}{Theorem}[section]
\newtheorem{proposition}[theorem]{Proposition}
\newtheorem{lemma}[theorem]{Lemma}
\newtheorem{corollary}[theorem]{Corollary}

\theoremstyle{definition}

\newtheorem{definition}[theorem]{Definition}
\theoremstyle{remark}
\newtheorem{remark}[theorem]{Remark}

\makeatletter
\let\@wraptoccontribs\wraptoccontribs
\makeatother

\contrib[with an appendix by]{Simon Brandhorst and Markus Kirschmer}
\address{Markus Kirschmer:
Fakult\"at f\"ur Mathematik, Universit\"at Bielefeld, Universit\"atsstra{\ss}e 25, 33501 Bielefeld, Germany.}
\email{mkirschm@math.uni-bielefeld.de}

\subjclass[2020]{11H56, 14J28}

\begin{document}

\maketitle
\begin{abstract}
    We give a classification of even integral lattices with virtually abelian reduced automorphism group. As a consequence, we complete the classification of K3 surfaces with virtually abelian automorphism group.
    In the appendix we formulate an algorithm for weak approximation in orthogonal groups and use it to determine if two indefinite lattices are isometric. 
\end{abstract}

\section{Introduction}

According to the Tits alternative \cite{tits}, any finitely generated subgroup of isometries of a hyperbolic lattice $L$ is either virtually solvable or it contains a free, non-abelian subgroup, where we say that a property of a group holds \emph{virtually} if it holds for a finite-index subgroup. In view of our geometric applications, we restrict to \emph{even} hyperbolic lattices, and we focus on a specific subgroup of isometries of $L$, namely its \emph{reduced automorphism group} (sometimes also called its \emph{symmetry group}). Let $\HH_L$ be the hyperbolic space associated to the hyperbolic lattice $L$. The \emph{Weyl group} $W(L)\subseteq \Or^+(L)$ is the subgroup generated by reflections into $(-2)$-roots of $L$, and we denote by $\mathcal{D}_L$ the closure of a fundamental domain for the action of $W(L)$ on the positive cone of $L$. Then the reduced automorphism group of $L$ can be equivalently defined as the quotient $\AutredL \coloneqq \Or^+(L)/W(L)$, or as the subgroup of $\Or^+(L)$ preserving $\mathcal{D}_L$.

Basic results in hyperbolic geometry show that the reduced automorphism group $\AutredL$ is virtually solvable if and only if it is virtually abelian, and this is the case exactly when $\AutredL$ is \emph{elementary}, i.e. it has a finite orbit in $\HH_L$ (cf. \cite[Theorems~5.5.9~and~5.5.10]{ratcliffe}). There are three types of elementary groups of isometries in $\Or^+(L)$ (cf. \cite[§5.5]{ratcliffe}):
\begin{enumerate}
\item Elementary groups \emph{of elliptic type} have a fixed point in the interior of $\HH_L$, and they are finite;
\item Elementary groups \emph{of parabolic type} have precisely one fixed point in the boundary $\partial \HH_L$, and they have zero entropy (cf. \cite[Theorem~3.7]{brandhorst.mezzedimi});
\item Elementary groups \emph{of hyperbolic type} have two fixed points in $\partial \HH_L$, and they are virtually cyclic.
\end{enumerate}

The main goal of this paper is to complete the classification of virtually solvable (or abelian) reduced automorphism groups of hyperbolic lattices. In rank $\le 2$ the reduced automorphism group $\AutredL$ is either finite or elementary of hyperbolic type (see e.g. \cite{galluzzi.lombardo.peters}). 
Our main result is the following:
\begin{theorem} \label{thm:main}
An even hyperbolic lattice of rank $\ge 3$ has a virtually solvable reduced automorphism group if and only if it belongs to an explicit list of $360$ lattices. More precisely, $\AutredL$ is:
\begin{enumerate}
    \item finite if and only if $L$ belongs to Nikulin--Vinberg's list of $118$ lattices. In particular $\rk(L)\le 19$.
    \item elementary of parabolic type if and only if $L$ belongs to the list in \cite{brandhorst.mezzedimi} of $194$ Borcherds lattices. In particular $\rk(L)\le 26$.
    \item elementary of hyperbolic type if and only if $L$ belongs to the list of $48$ lattices in \Cref{thm:classification.3} and \Cref{thm:classification.4}. In particular $\rk(L)\le 4$.
\end{enumerate}
\end{theorem}
Hyperbolic lattices of rank $\ge 3$ with finite reduced automorphism group were classified by Nikulin \cite{nikulin.finite.aut.greater.5, nikulin.finite.aut.3} and Vinberg \cite{vinberg.finite.aut.4}, while lattices with elementary reduced automorphism groups of parabolic type were recently classified in \cite{brandhorst.mezzedimi} (and independently by X. Yu in \cite{yu.entropy} in the case of Picard lattices of K3 surfaces).
By Nikulin \cite[Theorem 1]{nikulin.interesting}, the number of isomorphism classes of lattices of a fixed rank $n\geq 3$ with elementary reduced automorphism group of hyperbolic type is finite.

Thanks to the Torelli theorem, the automorphism group of a K3 surface $X$ over an algebraically closed field $k$ can be identified, up to a finite group, with the reduced automorphism group of its Picard lattice $\Pic(X)$. Therefore a byproduct of the classification in \Cref{thm:main} is the following:

\begin{corollary}
Let $X$ be a K3 surface over an algebraically closed field $k$. The automorphism group $\Aut(X)$ is virtually solvable if and only if either $\rho(X)=\rk(\Pic(X))\le 2$, or $\Pic(X)$ belongs to an explicit list of $359$ lattices of rank $\ge 3$. If this is the case, the rank of a maximal abelian subgroup of $\Aut(X)$ is at most $8$.
\end{corollary}

K3 surfaces with an elementary automorphism group of elliptic or parabolic type are the so-called K3 surfaces \emph{of zero entropy}, and they were studied in our previous paper \cite{brandhorst.mezzedimi}. From a dynamical point of view, K3 surfaces with an elementary automorphism group of hyperbolic type are more interesting: all their automorphisms of infinite order have positive entropy \cite[Theorem~5.5.8]{ratcliffe}, and therefore their elliptic fibrations (if they exist at all) have finite Mordell--Weil group (see \Cref{lem:detect_virtually_abelian}). K3 surfaces $X$ with the property that all elliptic fibrations on $X$ have finite Mordell--Weil group were first studied by Nikulin, see \cite[§3]{nikulin.elliptic}. If the Picard rank $\rho(X)$ is at least $6$, then this property is equivalent to having finite automorphism group (cf. \cite[Theorem~3.1]{nikulin.elliptic}). The same is true in rank $\rho(X)=5$ if one excludes two infinite series of lattices (cf. \cite[Theorem~3.2]{nikulin.elliptic}). The cases of rank $\rho(X)\le 4$, which were not treated by Nikulin, are the main focus of this paper.

We mention some closely related recent works: 
Kikuta and Takatsu proved that $\Aut(X)$ is a geometrically finite Kleinian group \cite{kikuta,takatsu}.
It follows that $\Aut(X)$ is either virtually abelian or non-elementary relatively hyperbolic relative to virtually abelian
subgroups of finite rank at least two. In case there are no such subgroups, $\Aut(X)$ is hyperbolic in the sense of Gromov \cite[Corollary 3.5]{fujiwara.oguiso.yu}.
Fujiwara, Oguiso and Yu showed finiteness of the list of Picard lattices of complex projective K3 surfaces of Picard rank at least $6$ with $\Aut(X)$ hyperbolic \cite{fujiwara.oguiso.yu}.

\subsection*{Strategy and outline}
It follows from a result of Nikulin \cite[Theorem 9.1.1]{nikulin.factor.groups} that, if a hyperbolic lattice of rank $\ge 5$ has an elementary reduced automorphism group, then it is either of elliptic or parabolic type (cf. \cite[Theorem~3.9]{brandhorst.mezzedimi}). Therefore, the core of our classification is in rank $3$ and $4$. Let $L$ by a hyperbolic lattice of rank $3$ or $4$ with an elementary reduced automorphism group of hyperbolic type. Then it is $2$-reflective of hyperbolic type in the sense of \cite[Definition 2]{nikulin.interesting}. By \cite[Lemma~1]{nikulin.interesting} we know that $L$ admits a $\QQ$-basis of $(-2)$-roots with bounded intersection multiplicity. This makes the classification in \Cref{thm:main} into a finite, though impractically long, task of checking all the possible lattices. Therefore, we develop three tests, the \emph{overlattice test} (see \Cref{lem:overlattice}), the \emph{sublattice test} (see \Cref{alg:sublattice}) and the \emph{cusp test} (see \Cref{alg:cusp test}) that allow us to exclude all but few lattices. We then employ the first author's implementation in OSCAR \cite{OSCAR} of Borcherds' method \cite{shimada.borcherds} as formulated by Shimada to obtain the classification in \Cref{thm:main}. Alternatively, one may use Allcock's algorithm developed in \cite{allcock.alternative}. Our computations also confirm the correctness of Nikulin's and Vinberg's classifications of hyperbolic lattices with finite reduced automorphism group and rank $3$ and $4$ in \cite{nikulin.finite.aut.3} and \cite{vinberg.finite.aut.4}. 

Let us briefly outline the contents of the paper. In \Cref{sec:preliminaries} we recall some known properties of hyperbolic lattices with an elementary reduced automorphism group of hyperbolic type, and we prove the simple, but crucial, \Cref{prop:easy_fundamental_chamber}. In \Cref{sec:tests} we introduce the three tests discussed above, and then in \Cref{sec:classification} we apply them to obtain the final classification. 
In \Cref{appendix} we give an algorithm realizing the weak approximation theorem for orthogonal groups and use it to determine if two indefinite lattices of rank at least $3$ are isometric. The algorithm is used in our classification to reduce computation times and to ensure that the final result consists of pairwise non-isometric lattices. 

\subsection*{Acknowledgments}
We warmly thank Serge Cantat for helpful discussions and comments, and the anonymous referees for comments that improved the structure of the paper.  This work started during the conference ``Birational geometry and dynamics'' at the SwissMAP Research Station Les Diablerets. We thank the organizers of the conference and the staff of the station for their hospitality.

\section{Preliminaries} \label{sec:preliminaries}

We keep notations as in our previous article \cite{brandhorst.mezzedimi}. We refer to \cite[§2]{brandhorst.mezzedimi} for the basic definitions about lattices, and only recall the notions that will be strictly necessary for our purposes. All lattices will always be assumed to be even.

Let $L$ be a hyperbolic lattice, and denote by $\HH_L$ the hyperbolic space associated to $L$. We choose a (closed) fundamental chamber $\D_L$ for the action of the Weyl group $W(L)$ on $L_\RR:=L \otimes \RR$, and, by a slight abuse of notation, we denote again by $\D_L$ the image of $\D_L$ in $\HH_L$. The chamber $\D_L$ is a locally polyhedral convex cone, whose walls are given by $r^\perp$, where $r$ runs along the simple $(-2)$-roots of $L$. The chambers $\{w(\D_L)\}_{w\in W(L)}$ are called \emph{Weyl chambers} of $L$, and they form a tessellation of $\HH_L$.

The group $\AutredL$ of isometries of $L$ preserving the fundamental chamber $\D_L$ is called the \emph{reduced automorphism group} of $L$. By construction it coincides with the quotient $\Or^+(L)/W(L)$. The reduced automorphism group $\AutredL$ is \emph{elementary} if it has a finite orbit in $\HH_L$. If $\AutredL$ is infinite, we say that it is \emph{elementary of hyperbolic type} if it has two fixed points in the boundary $\partial \HH_L$ (cf. \cite[§5.5]{ratcliffe}). In this case $\AutredL$ is virtually cyclic, and all the isometries in $\AutredL$ of infinite order have positive entropy \cite[Theorem~5.5.8]{ratcliffe}. Reduced automorphism groups of hyperbolic lattices of rank $\le 2$ are either finite or elementary of hyperbolic type (see \cite{galluzzi.lombardo.peters}). More precisely, a hyperbolic lattice of rank $2$ has finite reduced automorphism group if and only if it does not represent $0$ or $-2$ \cite[Corollary~3.4]{galluzzi.lombardo.peters}. We collect in the following lemma the known properties of hyperbolic lattices with an elementary reduced automorphism group of hyperbolic type.

\begin{lemma}\label{lem:detect_virtually_abelian}
Let $L$ be an even hyperbolic lattice with an elementary reduced automorphism group of hyperbolic type. Then:
\begin{enumerate}
\item[(1)] The rank of $L$ is at most $4$.
\item[(2)] If $0\ne v\in L$ is an isotropic vector, its stabilizer in $\AutredL$ is finite and $v^\perp/\langle v \rangle$ is a root lattice. 
\item[(3)] A finite-index sublattice of $L$ is spanned by $(-2)$-roots $\{\delta_i\}_{i=1,\ldots,\rk(L)}$ such that $-2 < \delta_i . \delta_j \le 18$ for all $i,j$.
\end{enumerate}
\end{lemma}
\begin{proof}\hfill
\begin{enumerate}
\item[(1)] It follows by \cite[Theorem~9.1.1]{nikulin.factor.groups} and the preceding discussion (also cf. \cite[Theorem~3.9]{brandhorst.mezzedimi}).
\item[(2)] It follows by \cite[Theorem~3.9~and~Proposition~3.2]{brandhorst.mezzedimi} that $v^\perp/ \langle v \rangle$ is an overlattice of a root lattice. Since it has rank at most $2$, and the root lattices of rank at most $2$, namely $A_1,A_2,A_1\oplus A_1$, have no non-trivial overlattices, it is a root lattice.
\item[(3)] This is \cite[Lemma~1]{nikulin.interesting}. \qedhere
\end{enumerate}
\end{proof}

We conclude this section with a crucial proposition, which is a refinement of \cite[Theorem~3.6]{mezzedimi.entropy}. We say that two groups $G_1,G_2$ are \emph{commensurable} if they admit isomorphic finite-index subgroups $H_1\subseteq G_1$, $H_2\subseteq G_2$.

\begin{proposition}\label{prop:easy_fundamental_chamber}
Let $L$ be an even hyperbolic lattice and $M\subseteq L$ a primitive sublattice of corank $1$. If $\disc(L)\ge 2\disc(M)$, then $M_\RR \cap \D_{L}$ is a Weyl chamber of $M$. In particular $\AutredL$ has a subgroup commensurable with $\AutredM$.
\end{proposition}
\begin{proof}
The fundamental chamber $\D_L$ of $L$ is cut out by hyperplanes given by the $(-2)$-roots of $L$, and therefore $M_\RR \cap \D_{L}$ is cut out by the hyperplanes $H_r \coloneqq r^\perp\otimes \RR$ given by $(-2)$-roots $r\in L$ of the form $r=x+y$, with $x\in M^\vee$ and $y\in N^\vee$, where $N\coloneqq M^\perp$. Note that $N$ is negative definite.
Since the Weyl chambers of $M$ are cut out by a subset of the $(-2)$-roots of $L$, we have that $M_\RR \cap \D_L$ is contained in a Weyl chamber of $M$, which we denote by $\D_M$.
If $x^2 \ge 0$, then $H_r$ does not intersect $\D_{M}^\circ$, since $M$ is a hyperbolic lattice. Therefore we may assume $x^2 < 0$, or equivalently $y^2>-2$.

In turn, the Weyl chamber of $M$ is cut out by a subset of such hyperplanes, namely those corresponding to $(-2)$-roots of $M$, or equivalently those with $y=0$. In particular $M_\RR \cap \D_{L}$ coincides with the Weyl chamber $\D_M$ if there is no $(-2)$-root $r=x+y\in L$ with $x\in M^\vee$, $y\in N^\vee$, and $y^2 \in (-2,0)$.

By contradiction, assume that there is such a $(-2)$-root $r=x+y$. The vector $y\in N^\vee$ belongs to the image $\pi(L)$ of the orthogonal projection $\pi: L \to N^\vee$. Set $n\coloneqq [L:M\oplus N] = [\pi(L):N]$, and let $-2k$ be the norm of the generator of the rank $1$ lattice $N$. We have \[\disc(L)\cdot n^2 = \disc(M)\disc(N)=2k \disc(M),\] so
$$-2 < y^2 \le \max(\pi(L)) = -2k/n^2  = -\frac{\disc(L)}{\disc(M)},$$
which implies that $\disc(L)< 2\disc(M)$, a contradiction.

For the second statement, observe that there exists a finite-index subgroup $G$ of $\Or(M)$ of isometries that extend to $L$. Then, by the first part, the injective map $G\hookrightarrow \Or(L)$ descends to an inclusion $G \cap \AutredM\hookrightarrow \AutredL$, giving the desired claim.
\end{proof}

The following lemma will allow us to easily apply \Cref{prop:easy_fundamental_chamber} in explicit examples.

\begin{lemma}\label{lem:easy_vectors}
 Let $L$ be an even lattice and $v \in L$ primitive. Set $k \ZZ = v.L$, $kg=|v^2|$ and $M=v^\perp \subseteq L$.
 Then 
 \[2\disc(M) \leq \disc(L) \quad \Leftrightarrow \quad 2g \leq k.\]
 Moreover, $k \mid \disc(L)$. 
\end{lemma}
\begin{proof}
We note that $g=[L:M \oplus \ZZ v]$. Then
\[\disc(L) g^2 = |v^2|\disc(M)=kg \disc(M)\]
so 
\[2\disc(M) = 2 \disc(L)g/k \leq \disc(L)\]
if and only if
\[2 g \leq k.\]
Since $g$ divides $\disc(M)$, we get that $\disc(L)=k \disc(M)/g$, i.e. $k \mid \disc L$. 
\end{proof}

\section{The tests} \label{sec:tests}

In this section we present the tests that we use to decide whether a given hyperbolic lattice of rank $3$ or $4$ has an elementary reduced automorphism group of hyperbolic type. The first one is the \emph{overlattice test}, and it is given by the following easy observation.

\begin{lemma}\label{lem:overlattice}
If $L$ has a virtually abelian reduced automorphism group and $L'$ is an overlattice of $L$, then $L'$ has virtually abelian reduced automorphism group as well.
Moreover, if $L$ has elementary reduced automorphism group of hyperbolic type, then the reduced automorphism group of $L'$ is either finite or elementary of hyperbolic type as well.
\end{lemma}
\begin{proof}
The subgroup $S$ of symmetries of $L'$ that preserve $L$ is of finite index in $\AutredLpr$.
We may choose the fundamental chamber $\mathcal{D}_L$ of $L$ to contain $\mathcal{D}_{L'}$, so in particular $S \subseteq \AutredL$. Therefore $S$ is virtually abelian and hence so is $\AutredLpr$.
If $\AutredL$ is elementary of hyperbolic type, then $S$ is either finite or elementary of hyperbolic type as well.
\end{proof}

The second test is called the \emph{sublattice test}. It is based on the following proposition, which in turn relies on \Cref{prop:easy_fundamental_chamber}.

\begin{proposition}\label{prop:rank2sublattice}
Let $L$ be an even hyperbolic lattice of rank $3$, and assume that there exist two \emph{non-isometric} primitive sublattices $M_1,M_2$ of $L$ of rank $2$, not representing $0$ or $-2$, such that $\disc(L)\ge 2\disc(M_i)$ for $i=1,2$. Then the action of $\AutredL$ on $\mathbb{H}_L$ is not elementary, and in particular $\AutredL$ is not virtually abelian.
\end{proposition}
\begin{proof}
By assumption, the reduced automorphism group of $M_1$ and $M_2$ is elementary of hyperbolic type, and in particular it is virtually cyclic. By \Cref{prop:easy_fundamental_chamber}, and up to replacing $M_2$ by $w(M_2)$ for some element $w \in W(L)$, we have $\D_{M_i} = M_i \otimes \RR \cap \D_{L}$.

Let $f_1,f_2$ be symmetries of infinite order of $M_1$ and $M_2$, respectively. \Cref{prop:easy_fundamental_chamber} implies that (some iterate of) $f_1$ and $f_2$ extend to symmetries of $L$, which, by abuse of notation, we still denote $f_1$ and $f_2$; we claim that the subgroup $\Gamma \coloneqq \langle f_1,f_2\rangle$ of $\AutredL$ is not elementary, so that $\AutredL$ is not elementary either.

If $\Gamma$ is elementary, there exists a vector $v\in \partial \overline{\mathbb{H}}_L$ whose orbits under $f_1$ and $f_2$ are finite. The only finite orbit of $f_1$ (resp. $f_2$) is given by the two points in $\partial \overline{\mathbb{H}}_{M_1}$ (resp. $\partial \overline{\mathbb{H}}_{M_2}$). These two vectors generate $(M_1)_\RR$ (resp. $(M_2)_\RR$), hence if $\Gamma$ is elementary, then $\partial \overline{\mathbb{H}}_{M_1} = \partial \overline{\mathbb{H}}_{M_2}$ and in particular $(M_1)_\RR = (M_2)_\RR$. Since $M_1$ and $M_2$ are both primitive sublattices of $L$, this implies that $M_1 = M_2$, which is in contradiction to the assumption that they are not isometric.
\end{proof}

Let $L$ be a lattice and $M,N \subseteq L$ primitive sublattices such that $\rk L = \rk M+ \rk N$. 
Then we call 
\[M\oplus N \subseteq L\]
a primitive extension. Primitive extensions are in bijection with anti-isometries
\[A_M \supseteq H_M \xrightarrow[]{\gamma} H_N \subseteq A_N.\]
where $|H_M|=|H_N|=[L:M \oplus N]$.
Alternatively, one may construct them from the chain of finite-index inclusions 
\[M\oplus N \subseteq L \subseteq L^\vee \subseteq M^\vee \oplus N^\vee.\]
\begingroup
\captionof{algorithm}{The sublattice test}\label{alg:sublattice}
\endgroup
\hrule
\begin{algorithmic}[1]
\REQUIRE An even hyperbolic lattice $L$ of rank $n$. An oracle which decides if a lattice of rank $n-1$ has finite or elementary reduced automorphism group of hyperbolic type (cf. \Cref{rmk:oracle}). 
\ENSURE If the output is ``false'', then $\AutredL$ is neither finite nor elementary of hyperbolic type.
 \STATE $\mathcal{M} = []$
 \FOR{$k \mid \disc(L)$}
  \FOR{$g \in \{g : 2g \leq k, 2 \mid gk \}$}
   \STATE Let $E$ be the set of primitive extensions
   \[[-gk]\oplus L(-1) \subseteq G\] 
   of index $k$.
   \FOR{$G \in E$}
     \STATE $D \gets A_G(-1)$
     \STATE Check if there exists a lattice $M$ of signature $(1,n-2)$ with discriminant group $A_M \cong D$. If not, continue with the next $G$, otherwise compute one such $M$.
     \STATE Compute a primitive extension $M \oplus G \subseteq \Lambda$ with $\Lambda$ unimodular.
     \STATE $L' \gets L(-1)^\perp \subseteq \Lambda$
     \STATE Check if $L \cong L'$. If not, continue with the next $G$.  
     \STATE Using the oracle, check if $\AutredM$ is finite or elementary of hyperbolic type. 
     \IF{$\AutredM$ is not finite or hyperbolic}
       \RETURN false \label{line:not finite or hyperbolic}
     \ENDIF
     \IF{$n=3$ and $M$ does not represent $0,-2$}\label{line:M-represents-0-2}
       \IF{$M$ is not isometric to a lattice in $\mathcal{M}$}
         \STATE append $M$ to $\mathcal{M}$.
     \ENDIF
       \IF{$|\mathcal{M}|=2$ }
         \RETURN false \label{line:two-test}
       \ENDIF
    \ENDIF
    \ENDFOR
  \ENDFOR
\ENDFOR
\RETURN true
\end{algorithmic}\hrulefill\\
\begin{remark}\label{rem.explain_sublattice_test}
The procedure in \Cref{alg:sublattice} computes primitive  hyperbolic sublattices $M$ of corank $1$ of $L$ with \[2\disc(M)\le \disc(L).\] Indeed, let $M$ be such a sublattice. Let $M^\perp = \langle v \rangle$, and set $k\ZZ = v.L$, $kg = |v^2|$. By \Cref{lem:easy_vectors} it follows that $k\mid \disc(L)$ and $2g \le k$, and $L$ is a primitive extension of $M\oplus [ -gk]$ of index $g$. The lattices $L$ and $L(-1)$ glue to a unimodular lattice, hence there is a primitive extension 
\[L \oplus L(-1) \subseteq \Lambda \cong U^{n}.\] 
Hence, under the composition \[ M \oplus [-gk] \oplus L(-1) \hookrightarrow L \oplus L(-1)\hookrightarrow \Lambda,\] $M$ is the orthogonal complement of a primitive extension $G$ of $[-gk]\oplus L(-1)$ of index $k$. 
\end{remark}

\begin{lemma}
\Cref{alg:sublattice} is correct.
\end{lemma}
\begin{proof}
The iteration over $k \mid \disc(L)$ and $g$ is justified by \Cref{lem:easy_vectors}.
By construction (cf. \Cref{rem.explain_sublattice_test}), we have that $M$ embeds primitively in $L$, with orthogonal complement $M^\perp \cong [-gk]$. By  
\Cref{lem:easy_vectors},
\[2 \disc(M) \leq \disc(L).\]
By \Cref{prop:easy_fundamental_chamber} the fundamental chamber $\D_M$ of $M$ is contained in some fundamental chamber $\D_L$ of $L$. Then $\AutredM\times \{\id_{M^\perp}\}$ is commensurable with a subgroup $\AutredL$. Thus if $\AutredM$ is not finite or elementary of hyperbolic type, then neither is $\AutredL$.
Thus the return value in line \ref{line:not finite or hyperbolic} is correct. 

After passing line \ref{line:M-represents-0-2}, $M$ does not represent $0,-2$ and therefore $\AutredM$ is elementary of hyperbolic type. Then line \ref{line:two-test} is justified by \Cref{prop:rank2sublattice}.
\end{proof}

\begin{remark}\label{rmk:oracle}
For $n=2$ the reduced automorphism group is always finite or elementary of hyperbolic type. 
The oracle for $n\ge 3$ can be obtained by induction as a finite list of lattices. 
\end{remark}
\begin{remark}
The computation of $E$ in \Cref{alg:sublattice} is the most expensive step. It suffices to compute it up to the action of $\Or([-gk])\times \Or(L^\vee/L(-1))$. This may or may not be faster than enumerating the primitive extensions straight away. 
\end{remark}

The third test, called the \emph{cusp test}, checks whether the lattice $L$ admits primitive isotropic vectors with infinite stabilizer. If this is the case, then the reduced automorphism group of $L$ is not elementary of hyperbolic type by \Cref{lem:detect_virtually_abelian}. Recall that a primitive isotropic vector $v\in L$ has infinite stabilizer if and only if the lattice $M=v^\perp/\langle v \rangle$ is not a root overlattice, i.e. $\rk(M_{root}) < \rk(M)$ (cf. \cite[Proposition~3.2]{brandhorst.mezzedimi}). 

\begingroup
\captionof{algorithm}{The cusp test}\label{alg:cusp test}
\endgroup
\hrule
\begin{algorithmic}[1]
\REQUIRE $L$ an even hyperbolic lattice of rank $n$.
\ENSURE If the output is ``false'', then $\AutredL$ is neither finite nor elementary of hyperbolic type.
\IF{$L$ is anisotropic}
\RETURN true
\ENDIF
\STATE compute a (random) isotropic vector $v \in L$ and set $M=v^\perp/[v]$
\STATE Let $R$ be the root sublattice of $M$.
\IF{$\rk R < n -2$}
\RETURN false 
\ENDIF 
\RETURN true
\end{algorithmic}\hrulefill\\

\begingroup
\captionof{algorithm}{Borcherds' test, cf. \cite{shimada.borcherds} and \cite[Appendix~A]{brandhorst.mezzedimi}}\label{alg:Borcherds' test}
\endgroup
\hrule
\begin{algorithmic}[1]
\REQUIRE An even hyperbolic lattice $L$ with an embedding into ${\rm II}_{1,25}\coloneqq U\oplus E_8^3$ such that $L^\perp \subseteq {\rm II}_{1,25}$ contains a $(-2)$ root.
\ENSURE Return if $\AutredL$ is finite or elementary hyperbolic.
\STATE Compute a set of generators $G$ for $\AutredL$ using Borcherds' method. 
\STATE Compute the fixed space of the matrices in $G$. If it is of signature $(1,*)$, then $\AutredL$ is finite, return true. If it is degenerate, then $\AutredL$ is parabolic of elementary type and we return false. 
\STATE If the rank of $L$ is at least $5$, return false.
\STATE Compute random products of the generators until one of positive entropy is found, and call it $f$. Compute the corresponding co-fixed lattice $C=(\ker f^{12} -1)^\perp$. \label{line:3}
\RETURN whether each generator of $G$ preserves $C$  
\end{algorithmic}\hrulefill\\
\begin{remark}
    In \Cref{alg:Borcherds' test} one can abort Borcherds' method as soon as one finds two non-commuting elements of positive spectral radius and return false. 
    If the algorithm reaches line \ref{line:3}, then $L$ is of positive entropy and the random search succeeds.
\end{remark}

\section{The classification} \label{sec:classification}
We are now ready to proceed with the classification of hyperbolic lattices with elementary reduced automorphism group of hyperbolic type.
We denote the isometry class of a lattice $L$ by $[L]$.
\begingroup
\captionof{algorithm}{}\label{alg:maximal}
\endgroup
\hrule
\begin{algorithmic}[1]
\REQUIRE An integer $n=3,4$ and an oracle which decides if a lattice of rank $n-1$ has finite or elementary reduced automorphism group of hyperbolic type (cf. \Cref{rmk:oracle}). 
\ENSURE 
A superset of the set of isomorphism classes of lattices of rank $n$ with finite or elementary reduced automorphism group of hyperbolic type. The set of maximal even ones with finite or elementary reduced automorphism group. 
\IF{$n=3$}  
\STATE Let $L_{a,b,c}$ denote the lattice with Gram matrix 
\[G_{a,b,c}=\begin{pmatrix}
    -2 &a &b\\
    a  &-2 &c\\
    b  &c & -2
\end{pmatrix}\]
$H_1 \gets \{ [L_{a,b,c}] : -1 \leq a \leq 2, -1 \leq b \leq c \leq 18, L_{a,b,c} \text{ is hyperbolic} \}.$ \label{alg: def H1 rk 3}
\ENDIF
\IF{$n=4$}
\STATE Let $L_{a,b,c,d,e,f}$ denote the lattice with Gram matrix 
\[G_{a,b,c,d,e,f}=\begin{pmatrix}
    -2 & a &  b &  c\\
    a  &-2 &  d &  e\\
    b  & d & -2 &  f\\
    c  & e &  f & -2\\
\end{pmatrix}\]
\STATE $H_1 \gets \{ [L_{a,b,c,d,e,f}] \text{ hyperbolic}: \  -1 \leq a \leq 2, \ a \leq b,c,d,e,f \leq 18, \ b\leq c, d, e \}$ \label{alg: def H1 rk 4}
\ENDIF
\STATE $M_1 \gets \{[M] : L \subseteq M \text{ is a maximal even overlattice, } [L] \in H_1\}$ \label{alg:def M1}
\STATE $M_2 \gets \{[M] \in M_1 : M \text{ passes the cusp and sublattice tests} \}$ \label{alg:def M2}
\STATE Use Borcherds' test, Algorithm \ref{alg:Borcherds' test}, to compute \\
$M_3 \gets \{[M] \in M_2 : \AutredM\mbox{ is finite or elementary of hyperbolic type}\}$ 
\RETURN $M_3$, $H_1$
\end{algorithmic}\hrulefill\\
\begin{lemma}
\Cref{alg:maximal} is correct.
\end{lemma}
\begin{proof}
Let $L$ be a lattice with finite or elementary reduced automorphism group of hyperbolic type. 
If $n = \rk L =3$, by \cite[Lemma 1]{nikulin.interesting} there are $3$ roots $\delta_1,\delta_2,\delta_3 \in L$ which are linearly independent and such that \[-2 < \delta_i.\delta_j \leq 18.\] 
By inspection of the proof of this lemma, we can moreover assume that $a = \delta_1.\delta_2 \in \{-1,0,1,2\}$. Denote the Gram matrix of $(\delta_1,\delta_2,\delta_3)$ by $G_{a,b,c}$
with $-1 \leq a\leq 2$ and $-1\leq b,c \leq 18$. After possibly swapping $\delta_1$ and $\delta_2$, we may assume that $b \leq c$. Let $L_{a,b,c}$ denote the lattice with Gram matrix $G_{a,b,c}$. 

If $n= \rk L =4$, by \cite[Lemma 1]{nikulin.interesting} there are $4$ roots and by inspection of the proof we can once more assume that $-1 \leq \delta_i.\delta_j \leq 2$ for at least one pair $i,j$. 
After permuting $\delta_1,\dots,\delta_4$, we may assume that $a = \min(a,b,c,d,e,f)$. 
Then we may still exchange $\delta_1$ and $\delta_2$ as well as $\delta_3,\delta_4$ and achieve that 
$b \leq c,d,e$.
By construction, the set $H_1$ contains a sublattice of $L$ of finite index.
\end{proof}

\begin{remark}
A crucial part of this algorithm is to test lattices for isometry in order to create sets of isometry classes. See the appendix for an algorithm to do this. Testing the $n(n-1)/2$-pairs of a list of length $n$ for isometry is obviously too expensive if $n$ is in the range of millions. Therefore, we implemented a hash function for an isometry class that captures some invariants. The genus symbol \cite[Chapter 15]{conway.sloane.lattices} is a good invariant for this purpose. At the odd primes it is canonical and at the even ones, we simply take the ranks, scales and parity of the Jordan constituents, which is an invariant as well.  
\end{remark}

\begingroup
\captionof{algorithm}{Classification algorithm}\label{alg:classification}
\endgroup
\hrule
\begin{algorithmic}[1]
\REQUIRE \begin{tabular}[t]{@{\textbullet\hspace{0.5em}} p{0.928\linewidth} @{}}
    An integer $n$ and an oracle which decides if a lattice of rank $n-1$ has finite or elementary reduced automorphism group of hyperbolic type (cf. \Cref{rmk:oracle}). \\
    A finite set $H_1$ of candidates such that $[L] \in H_1$ for all even lattices $L$ of rank $n$ with $\AutredL$ finite or elementary of hyperbolic type. \\
    The set $M_3$ of isometry classes of all maximal even lattices $M$ of rank $n$ with $\AutredM$ finite or elementary of hyperbolic type.
\end{tabular}
\ENSURE A complete list of representatives of the isomorphism classes of lattices of rank $n$ with finite or elementary reduced automorphism group of hyperbolic type.
\STATE $H_2 \gets \{[L_1] : L \subseteq L_1 \subseteq M$ where $[L] \in H_1$ and $[M] \in M_3\}$ \label{alg:def H2}
\STATE sort $H_2$ by discriminant in ascending order and remove all maximal lattices from it. \label{alg:sort H2}
\STATE $H_3 \gets M_3$ \label{alg:def H3}
\FOR{$L \in H_2$} \label{alg:loop H2}
\FOR{$p \mid \disc(L)$ with $p^2 \mid \disc(L)$} \label{alg:loop indexp}
\STATE compute an even overlattice $L \subseteq L'$ of index $p$ if it exists, otherwise continue 
the loop in line \ref{alg:loop indexp} with the next prime $p$ 
\IF{$[L'] \notin H_3$}\label{alg: overlattice test}
\STATE continue the loop in line \ref{alg:loop H2} with the next lattice $L$. 
\ENDIF 
\ENDFOR
\IF{$L$ does not pass the cusp or sublattice tests}
\STATE continue the loop in line \ref{alg:loop H2} with the next lattice $L$ 
\ENDIF
\STATE append $L$ to $H_3$ if $L$ passes Borcherds' test, \Cref{alg:Borcherds' test}
\ENDFOR
\RETURN $H_3$
\end{algorithmic}\hrulefill\\
\begin{lemma}
\Cref{alg:classification} is correct.
\end{lemma}
\begin{proof}
Let $L$ be a lattice with finite or elementary reduced automorphism group of hyperbolic type. Then $[L] \in H_1$. 
By \Cref{lem:overlattice} every overlattice $L'$ of $L$ has finite or elementary reduced automorphism group of hyperbolic type. In particular, so does the maximal one. Hence $[L]$ is contained in $H_2$.
By induction and since $H_2$ is sorted by discriminant, the isometry class $[L']$ is already contained in $H_3$. Therefore, $L$ passes the overlattice test in line \ref{alg: overlattice test}.

The rest of the algorithm is clear. 
\end{proof}

We obtain the following results from running \Cref{alg:maximal,alg:classification} with $n=3,4$. We denote by $U(n,k)$ the rank $2$ hyperbolic lattice with Gram matrix
$$\begin{pmatrix}
    0 &n\\
    n &2k
\end{pmatrix}$$
and $U(n)\coloneqq U(n,0)$.

\begin{theorem} \label{thm:classification.3}
    Let $L$ be an even hyperbolic lattice of rank $3$ whose reduced automorphism group is elementary of hyperbolic type. Then $L$ is isometric to exactly one of the following $45$ lattices
    \[[2n] \oplus A_2 \qquad n \in \{5, 8, 10, 11, 14, 15, 20, 24, 26, 30, 42, 48, 60, 90\}\]
    
    \[[2n]\oplus A_1 \oplus A_1 \qquad n \in \{6, 7, 12, 15, 24, 36\}\]
    
    \[\begin{pmatrix}
        2l & 1 & 0\\
        1 & -2 & 1\\
        0 & 1 & -2\\
    \end{pmatrix} \qquad l\in \{11, 13, 15, 21, 23, 33\}\]
    
     \[\begin{pmatrix}
        2l & 1 & 0\\
        1 & -2 & 0\\
        0 & 0 & -2\\
    \end{pmatrix} \qquad l\in \{5,12\}\] 
     
     \[\begin{pmatrix}
        2l & 1 & 1\\
        1 & -2 & 0\\
        1 & 0 & -2\\
    \end{pmatrix} \qquad l\in \{3,7\}\] 
    \[U(n) \oplus [-2] \qquad n \in \{11,14,15,20,24\}\]

    \[U(n,k)\oplus [-2] \qquad (n,k) \in \{(8,2),(8,3),(8,6),(9,3),(12,3),(12,6),(12,8),(16,8)\}\]

    \[\begin{pmatrix}
        -2&3&3 \\
        3& -2 &2 \\
        3&  2& -2 
    \end{pmatrix},\begin{pmatrix}
        -2&3&7 \\
        3& -2 &2 \\
        7&  2& -2 
    \end{pmatrix}\]
\end{theorem}
\begin{remark}
Let $n=3$. The set $H_1$ consists of $403$ isometry classes.
The sets $M_1,M_2,M_3$ consist of $192,24,21$ isometry classes.
The set $H_2$ consists of $225$ isometry classes. 
Of these, $59$ are discarded by the overlattice test in Line \ref{alg: overlattice test}, $15$ by the cusp test, $62$ by the sublattice test and, finally, $18$ by Borcherds' method.
The resulting set $H_3$ has $71$ elements. Of these $26$ have finite reduced automorphism group, confirming Nikulin's result \cite{nikulin.finite.aut.3}, and $45$ have elementary reduced automorphism group of hyperbolic type. See \cite{intermediate.data} for the code and intermediate results.
\end{remark}

\begin{theorem} \label{thm:classification.4}
Let $L$ be an even hyperbolic lattice of rank $4$ with elementary reduced automorphism group of hyperbolic type. 
Then $L$ is isometric to exactly one of the following $3$ lattices
\[U(6) \oplus A_1 \oplus A_1, 
\begin{pmatrix} 
0 & 4 & 0 & 0\\
 4 & -6 & 1 & 1\\
 0 & 1 & -2 & 0\\
 0 & 1 & 0 & -2
\end{pmatrix},
\begin{pmatrix} 
0 & 5 & 0 & 0\\
 5 & -6 & 2 & 2\\
 0 & 2 & -2 & 1\\
 0 & 2 & 1 & -2
\end{pmatrix}
\]
of discriminants $2^4 \cdot 3^2$, $2^6$ and $3 \cdot 5^2$.
\end{theorem} 
\begin{remark}
Let $n=4$. 
The set $H_1$ consists of $71299$ isometry classes coming from $1,512,137$ hyperbolic Gram matrices.  
The set $M_1$ consist of $28746$ isometry classes. After discarding $26457$ with the cusp test and $2281$ with the sublattice test, $|M_2|=8$ remain. Of these, $|M_3|=4$ pass Borcherds' method. 

The set $H_2$ consists of $2759$ isometry classes,
of these $229$, $91$ and $21$ pass the overlattice, cusp and then sublattice test. After applying Borcherds' method, $|H_3|=17$ remain. Of these, $14$ have finite reduced automorphism group, as in \cite{vinberg.finite.aut.4}. See \cite{intermediate.data} for the code and intermediate results.
\end{remark}

\begin{remark}
We obtain a list of $48$ lattices with elementary hyperbolic reduced automorphism group in rank $3$ and $4$. Clearly, all $3$ lattices in rank $4$ are isotropic. On the other hand, observe that one of the lattices $L$ of rank $3$ in \Cref{thm:classification.3} is isotropic if and only if $2\disc(L)$ is a square. Indeed, if $v\in L$ is isotropic, then $v^\perp/\langle v \rangle \cong A_1$ by \Cref{lem:detect_virtually_abelian}, so $\disc(L) = 2n^2$, where $n=v.L$ is the divisibility of $v$. Conversely, if $r\in L$ is a $(-2)$-root, then $\disc(r^\perp)$ is a square and thus it represents zero. A simple computation shows that $20$ out of the $45$ lattices of rank $3$ in \Cref{thm:classification.3} are isotropic.
\end{remark}

\appendix

\section{An algorithm for weak approximation and an isometry test}\label{appendix}

In this section we give an algorithm to determine if two non-degenerate indefinite integer lattices of rank at least $3$ are isometric. The algorithm is feasible in practice even for larger ranks provided that the discriminant can be factored. It does not produce an explicit isometry.\\
Building on work of Plesken--Souviginer \cite{plesken.souvignier} and Mertens \cite{mertens2014}, Hulek and Dutour-Sikirić \cite{hulek-doutour2024} formulate an algorithm to compute an explicit isometry between arbitrary non-degenerate lattices. It has no theoretical limitations but it is feasible only for moderately small ranks and discriminants.\\

The difficulty for an isometry test is the following: in order to decide if two lattices $L_1$ and $L_2$ in the same genus are in the same spinor genus (and therefore isometric), one has to produce an isometry 
$f\colon L_1 \to L_2 \otimes \QQ$ such that $[L_2:f(L_1) \cap L_2]$ is coprime to $2\det(L_1)$.
\Cite{chan.gao.weak.approx} shows that finding such a proper isometry $f$ can be done by an exhaustive search.
Following \cite[101:7]{meara.quadratic}, our approach is based on splitting local isometries into reflections and approximating them globally.

This allows to effectively decide isometry of any two integer lattices:
The case of rank $1$ is trivial and rank $2$ is solved by Gauss' reduction of binary forms. The  definite case is checked with \cite{plesken.souvignier}. Finally, two degenerate lattices are isometric if and only if they have the same rank (as abelian groups) and $L/\ker(L) \cong L'/\ker(L')$.

The algorithms are implemented as part of the computer algebra system OSCAR \cite{OSCAR}.\\

Recall that, for $V$ a rational quadratic space and $x \in V$ with $x^2 \neq 0$, the reflection in $x$ is defined by 
$s_x(y) = y- \frac{2(x.y)}{x^2} x$.
For the estimates, we use the normalized $p$-adic absolute value with $|p|_p=1/p$ and the maximum norm $|| \cdot ||_p$ on endomorphisms, matrices and vectors. The normalized $p$-adic valuation is denoted by $\nu_p(x)$ and satisfies $\nu_p(p)=1$.
If confusion is unlikely, we drop $p$ from the indices.

We recall the following definition from \cite[Section 4]{brandhorst.veniani.hensel-lifting}:
\begin{definition}
Let \(f\colon L \to M\) be a linear map of \(\ZZ_p\)-lattices of the same rank. 
We say that \(f\) is \emph{compatible} if \(f(L \cap p^iL^\vee)=M \cap p^iM^\vee\) for all \(i \in \ZZ\).
For an integer \(a \geq 0\), we call a compatible linear map \(f\) \emph{\(a\)-approximate} if there exists an isometry \(\tilde f \colon L \to M\) such that \(\tilde f \equiv f \bmod p^a\mathrm{Hom}(L,M)\). 
\end{definition}
If $L$ and $M$ are \(\ZZ_p\)-lattices and 
$f\colon L \to M\otimes \QQ_p$ is a $\ZZ_p$-linear map, then we say that $f$ is $a$-approximate if $f\colon L \to f(L)$ is $a$-approximate. Conditions to be $a$-approximate in terms of matrices are spelled out in \cite[Definition 4.2]{brandhorst.veniani.hensel-lifting}.

\begingroup
\captionof{algorithm}{Weak approximation}\label{alg:weak approximation}
\endgroup
\hrule
\begin{algorithmic}[1]
\REQUIRE \begin{tabular}[t]{@{\textbullet\hspace{0.5em}} p{0.928\linewidth} @{}}
A non-degenerate quadratic space $V=\QQ^n$; \\
a finite set of primes $P$;\\
a collection of natural numbers $(t_p)_{p \in P}$;\\
a collection of matrices $(f_p)_{p \in P}$ with $f_p\in \mathrm{GL}(V \otimes \QQ_p)$ 
such that $f_p$ is $(\nu_p(d_p)+t_p)$-approximate, where $d_p$ is the denominator of $f_p$ and $\nu_p(\det(f_p)-1)\geq 1+\nu_p(4)$. 
\end{tabular}
\ENSURE Return a matrix $f \in \mathrm{SO}(V)$ with $||f - f_p||_p \leq p^{-t_p}$ for all $p \in P$.
\FOR{$p \in P$}\label{alg:weak approximation-lfirstloop}
  \STATE $s \gets t_p+\nu_p(d_p)$
  \WHILE{true} 
  \STATE Use the algorithm described in \cite[Lemma 5.8]{shimada.connected.components} to (try to) decompose $f_p$ as a series of $p$-adic reflections, i.e. compute $x_{1,p},\dots ,x_{r_p,p} \in \ZZ^n$ such that $f_p \equiv s_{1,p} \circ \dots \circ s_{r_p,p} \pmod{ p^{t_p}}$ where $s_{i,p}$ denotes the reflection in $x_{i,p}$. \label{alg:weak approximation-lshimada} If the decomposition was successful, break the while loop and continue with the for loop in Line \ref{alg:weak approximation-lfirstloop} with the next prime. 
  \STATE $s \gets 2s$
  \STATE Use quadratic Hensel lifting of \cite[Algorithm 3]{brandhorst.veniani.hensel-lifting} to compute a matrix $f_p' \in \QQ^{n \times n}$ such that $f_p'$ is $s$-approximate and $f_p \equiv f_p' \mod p^s f_p \ZZ^{n \times n}$.  
  \STATE $f_p \gets f_p'$
  \ENDWHILE
\ENDFOR
\STATE Pick some anisotropic vector $v \in V$.
\STATE For $r_p<i\leq r:=\max\{r_p : p \in P\}$ set $x_{i,p}=v$. \label{alg:weak-approximation.padding}
\STATE $(l_i \gets 2n+10 \colon$  for $i = 1,\dots n)$
\REPEAT
    \STATE For each $1 \leq i \leq r$ compute $x_i \in \ZZ^n$ with \[x_i \equiv x_{i,p} \pmod{p^{t_p+l_i}} \quad ( \forall p \in P).\]
    \vspace*{-4mm}
    \STATE $f\gets s_{x_1} \circ \dots \circ s_{x_n}$
    \FOR{ $p \in P$}
        \STATE $e_p \gets \log_p(||f - f_p||^{-1})$
        \IF{$e_p < t_p$}
        \STATE $l_i\gets l_i+10$
        \ENDIF
    \ENDFOR
\UNTIL{$e_p \geq t_p$ for all $p \in P$}
\RETURN $f$
\end{algorithmic}\hrulefill\\

\begin{proposition}
\Cref{alg:weak approximation} terminates and is correct. 
\end{proposition}
\begin{proof}
If $f_p$ is $a$-approximate and $a$ is sufficiently high, then the decomposition into reflections in Line \ref{alg:weak approximation-lshimada} is successful. 
Since $f_p$ is $a$-approximate, it is a valid input for the Hensel lifting algorithm.
From $s \geq t_p + \nu_p(d_p)$ we get $p^sf_p \ZZ^{n \times n} \subseteq p^{t_p}\ZZ^{n\times n}$. In particular, $||f_p-f_p'||_p \leq p^{-t_p}$.   

Thus, after the for loop in Line \ref{alg:weak approximation-lfirstloop}, $||f_p - s_{x_1,p} \circ \dots \circ s_{x_{r_p},p}||_p\leq p^{-t_p}$ holds for each $p \in P$.  
Since $\nu_p(\det(f_p)-1)\geq 1 + \nu_p(4)$, the number of reflections $r_p$ is even. Therefore, we only composed with the identity in Line \ref{alg:weak-approximation.padding} and 
\[||f_p - s_{x_1,p} \circ \dots \circ s_{x_{r},p}||\leq p^{-t_p} \quad (\forall p \in P)\]
holds. 

Once the $x_i$ approximate the $x_{i,p}$ sufficiently well, the $s_{x_i}$ approximate the $s_{x_i,p}$ because the dependence of $s_{x_i}$ on $x_i$ is continuous. Since matrix multiplication is continuous, the product of the $s_{x_i}$ eventually approximates $f_p$ and the algorithm terminates. 
\end{proof}

For the sake of completeness we give an apriori estimate for the precisions $l_i$ needed for the Chinese remainder. In practice, it will not be optimal. Hence, we included the while loop. The reader not interested in these estimates can safely jump to \Cref{alg:isometry-test}.

\begin{lemma}\label{lem:error}
Let $V=\QQ^n$ be a quadratic space with Gram matrix $G$ and $x \in V$ anisotropic. Assume that $||x||=||G||=1$, and set
\[ c = 1/\min\{|x^2|,|x^2|^2\}\]
and 
\[d = 1/ \min\{1, |x^2|\}.\]
If $h \in V$ with $||h|| < \min(1, |x^2|)$ then 
$ || s_x - s_{x+h}|| \leq c \cdot ||h||$
and 
\(||s_x|| = ||s_{x+h}|| \leq d. \)
\end{lemma}
\begin{proof}
\[\begin{array}{ll}
    & s_{x}(y) - s_{x+h}(y) = \dfrac{2 (x+h,y)}{(x+h)^2}(x+h) - \dfrac{2 (x,y)}{x^2}x\\[15pt]
    =& 2\left(\dfrac{1}{(x+h)^2}-\dfrac{1}{x^2} \right)(x,y) x
     + 2\dfrac{(h,y)}{(x+h)^2}x    + 2\dfrac{(h,y)}{(x+h)^2}h
     + 2\dfrac{(x,y)}{(x+h)^2}h 
\end{array}\]
We may estimate the error norm by the maximum over each summand and use that $|(x,y)|\leq ||x|| \cdot ||G|| \cdot ||y|| \leq ||y||$. 
Then we arrive at 
\[||2y|| \cdot  |(x+h)^2|^{-1} \cdot \max \left\{|x^2|^{-1}|(x+h)^2 - x^2| ,||h||, ||h||^2, ||h||\right\},\]
which simplifies to 
\[||2y|| \cdot |(x+h)^2|^{-1} \cdot \max \left\{|x^2|^{-1}|2x.h + h^2| , ||h||\right\}\]
and is bounded above by
\[|2| \cdot ||y|| \cdot  |(x+h)^2|^{-1} ||h|| \max \left\{ |x^2|^{-1}|2|, |x^2|^{-1} ||h|| , 1 \right\}\]
Note that by our assumptions on $||h||$, we have $|(x+h)^2| = |x^2|$. 
With $|2|, ||h|| \leq 1$ we arrive at the upper bound
\[||y|| \cdot ||h|| \cdot |x^2|^{-1} \max \left\{ |x^2|^{-1}, 1 \right\} =||y|| \cdot ||h|| \cdot  c  \]
The claim follows.
The first claim follows. For the second one, since $||h||<|x^2| \le \|Gx||$, we have
\[
||x+h||=||x||=1,\qquad |(x+h)^2|=|x^2|,\qquad ||G(x+h)\|=||Gx\|.
\]
For any anisotropic $y\in V$,
\[
s_y=I-\frac{2}{y^2}y(Gy)^T,
\]
so, since $s_y^2=I$,
\[
||s_y||=\max\left\{1,\frac{|2|\cdot ||y||\cdot ||Gy||}{|y^2|}\right\}.
\]
Thus the three equalities above give
\[
\|s_x\|=\|s_{x+h}\|
\le \max\{1,|x^2|^{-1}\}
=\frac1{\min\{1,|x^2|\}}=d.
\]
\end{proof}
\begin{proposition}
    Define $x_i$, $h_i$, $c_i$ and $d_i$ as in \Cref{lem:error} with the $x_i$, $h_i$ taking the places of $x$ and $h$. Then
\[|| s_{x_1+h_1} \circ \dots \circ s_{x_n+h_n} - s_{x_1}\circ \dots \circ s_{x_n}|| \leq  \max\left\{ \prod_{i \in I} c_i ||h_i|| \prod_{j \in I^c} d_j \mid \emptyset \neq I \subseteq \{1,\dots n\}\right\} \leq C e \]
where $e = \max\{||h_i|| : i \in \{1,\dots n\}\}$ 
and $C = \prod_{i=1}^n\max\{c_i, d_i\}$.
\end{proposition}
\begin{proof}
We write $s_{x_i+h} = s_{x_i}+r_i$ with $r_i = s_{x_i+h}-s_{x_i}$ in the endomorphism ring . By \Cref{lem:error}, $||s_{x_i}|| \leq d_i$ and $||r_i|| \leq c_i ||h_i||$. Then expand the product as a sum and estimate each term using the triangle inequality. The last inequality uses that $||h_i||\leq 1$.
\end{proof}

Given a $p$-adic lattice $L$, we say that $x \in L$ is a \emph{norm generator}, if $x^2\ZZ = n(L):=\sum_{y \in L} y^2 \ZZ$, which is called the \emph{norm} of $L$. The scale of $L$ is the (fractional) ideal $s(L):=\{x.y : x ,y \in L\}$. Since 
$2s(L)\subseteq n(L)$, the reflection in a norm generator preserves the lattice. 

\begingroup
\captionof{algorithm}{Isometry test}\label{alg:isometry-test}
\endgroup
\hrule
\begin{algorithmic}[1]
\REQUIRE $L_1,L_2$ indefinite integral lattices of rank $n \geq 3$.
\ENSURE Return whether $L_1$ and $L_2$ are isometric.
\IF{$L_1$ and $L_2$ are not in the same genus}
\RETURN false 
\ENDIF
\IF{the genus of $L_1$ consists of a single spinor genus}
\RETURN true 
\ENDIF 
\STATE Compute an isometry $f \colon L_1\otimes \QQ \to L_2\otimes \QQ$
\STATE $L_1 \gets f(L_1)$
\STATE $r \gets [L_1 : L_1 \cap L_2]$
\STATE $P = \text{support}(2 \det(L_1))$
\STATE Let $V:= L_2 \otimes \QQ = \QQ^n$, denote by $B_1$ the basis matrix of $L_1 \subseteq V$.
\STATE Let $d$ be the denominator of $B_1$.
\STATE $k_p \gets  \nu_p(r) + 2n + 10$
\FOR{$p \in P$} \label{alg:isometry-test-line-redo}
\STATE Compute a matrix $F_p \in \QQ$ such that $F_p$ is $k_p+\nu_p(d)$-approximate (w.r.t. $V$) and $F_p B_1 \in \ZZ_{(p)}$. This can be done using a $p$-adic normal form as described in \cite{miranda.morrison.embeddings} and quadratic Hensel lifting \cite[Algorithm 3]{brandhorst.veniani.hensel-lifting}. 
\IF{$\nu_p(\det(F_pB_1)-1) < 1 + \nu_p(4)$}
\STATE Compose $F_p$ by a reflection in a norm generator of $L_1 \otimes \ZZ_p$.
\ENDIF 
\ENDFOR
\STATE Compute an isometry $F \in \mathrm{SO}(V)$ with $F\equiv F_p \mod p^{\nu_p(r)}$ for all $p \in P$ using the weak approximation \Cref{alg:weak approximation}. 
\STATE $r \gets [L_2 : F(L_1) \cap L_2]$
\RETURN Return whether $r$ is an improper spinor generator of $L_1$ (see e.g. \cite[Chapter 15 \S 9]{conway.sloane.lattices}).
\end{algorithmic}\hrulefill\\

\begin{remark}
The correctness of \Cref{alg:isometry-test} follows from the theory of genera and spinor genera, see e.g. \cite{conway.sloane.lattices,cassels.quadratic,meara.quadratic}. The computation of the genus and the improper spinor generators is fast and effective, provided the factorization of the determinant is known. It is available for instance in \cite{OSCAR}.
\end{remark}

\printbibliography

@article {nikulin.elliptic,
    AUTHOR = {Nikulin, Viacheslav V.},
     TITLE = {Elliptic fibrations on {$\rm K3$} surfaces},
   JOURNAL = {Proc. Edinb. Math. Soc. (2)},
  FJOURNAL = {Proceedings of the Edinburgh Mathematical Society. Series II},
    VOLUME = {57},
      YEAR = {2014},
    NUMBER = {1},
     PAGES = {253--267},
      ISSN = {0013-0915},
   MRCLASS = {14J28 (14C22 14J27 14J50)},
  MRNUMBER = {3165023},
MRREVIEWER = {Marco Andreatta},
       DOI = {10.1017/S0013091513000953},
       URL = {https://doi.org/10.1017/S0013091513000953},
}

@book {ratcliffe,
    AUTHOR = {Ratcliffe, John G.},
     TITLE = {Foundations of hyperbolic manifolds},
    SERIES = {Graduate Texts in Mathematics},
    VOLUME = {149},
   EDITION = {Second},
 PUBLISHER = {Springer, New York},
      YEAR = {2006},
     PAGES = {xii+779},
      %ISBN = {978-0387-33197-3; 0-387-33197-2},
   MRCLASS = {57M50 (20H10 30F40 51M10)},
  MRNUMBER = {2249478},
}

@article {mezzedimi.entropy,
    AUTHOR = {Mezzedimi, Giacomo},
     TITLE = {K3 surfaces of zero entropy admitting an elliptic fibration
              with only irreducible fibers},
   JOURNAL = {J. Algebra},
  FJOURNAL = {Journal of Algebra},
    VOLUME = {587},
      YEAR = {2021},
     PAGES = {344--389},
      ISSN = {0021-8693},
   MRCLASS = {14J28 (14J27 14J50)},
  MRNUMBER = {4304793},
MRREVIEWER = {Jing Zhang},
       DOI = {10.1016/j.jalgebra.2021.08.005},
       URL = {https://doi.org/10.1016/j.jalgebra.2021.08.005},
}

@misc{OSCAR,
  key          = {OSCAR},
  organization = {The OSCAR Team},
  title        = {OSCAR -- Open Source Computer Algebra Research system,
                  Version 1.11.5},
  year         = {2025},
  Howpublished = {\url{https://oscar.computeralgebra.de}},
  }

@book{miranda.morrison.embeddings,
 Author = {Rick {Miranda} and David R. {Morrison}},
 Title = {Embeddings of Integral Quadratic Forms.},
  Note = {\url{https://web.math.ucsb.edu/~drm/manuscripts/eiqf.pdf}}, 
%Note = {\href{https://web.math.ucsb.edu/~drm/manuscripts/eiqf.pdf}{https://web.math.ucsb.edu/~drm/manuscripts/eiqf.pdf}}
}

@article {hulek-doutour2024,
    AUTHOR = {Dutour Sikiri\'c, Mathieu and Hulek, Klaus},
     TITLE = {Moduli of polarised {E}nriques surfaces---computational
              aspects},
   JOURNAL = {J. Lond. Math. Soc. (2)},
  FJOURNAL = {Journal of the London Mathematical Society. Second Series},
    VOLUME = {109},
      YEAR = {2024},
    NUMBER = {1},
     PAGES = {Paper No. e12828, 32},
      ISSN = {0024-6107,1469-7750},
   MRCLASS = {14J28 (11H55 14J10 14Q10 20G30)},
  MRNUMBER = {4680217},
MRREVIEWER = {Aline\ Zanardini},
       DOI = {10.1112/jlms.12828},
       URL = {https://doi.org/10.1112/jlms.12828},
}

@article{takatsu,
   title={Blown-up boundaries associated with ample cones of K3 surfaces},
   ISSN={2152-6842},
   url={http://dx.doi.org/10.1017/nmj.2025.10082},
   DOI={10.1017/nmj.2025.10082},
   journal={Nagoya Mathematical Journal},
   publisher={Cambridge University Press (CUP)},
   author={Takatsu, Taiki},
   year={2025},
   month=sep, pages={1–30} }

@article {mertens2014,
    AUTHOR = {Mertens, Michael H.},
     TITLE = {Automorphism groups of hyperbolic lattices},
   JOURNAL = {J. Algebra},
  FJOURNAL = {Journal of Algebra},
    VOLUME = {408},
      YEAR = {2014},
     PAGES = {147--165},
      ISSN = {0021-8693,1090-266X},
   MRCLASS = {11H56},
  MRNUMBER = {3197177},
MRREVIEWER = {Mikhail\ V.\ Belolipetsky},
       DOI = {10.1016/j.jalgebra.2013.09.034},
       URL = {https://doi.org/10.1016/j.jalgebra.2013.09.034},
}

@misc{kikuta,
      title={Geometrical finiteness for automorphism groups via cone conjecture}, 
      author={Kohei Kikuta},
      year={2024},
      eprint={2406.18438},
      archivePrefix={arXiv},
      primaryClass={math.AG},
      url={https://arxiv.org/abs/2406.18438}, 
}

@misc{fujiwara.oguiso.yu,
      title={On K3 surfaces with hyperbolic automorphism groups}, 
      author={Koji Fujiwara and Keiji Oguiso and Xun Yu},
      year={2025},
      eprint={2507.13726},
      archivePrefix={arXiv},
      primaryClass={math.AG},
      url={https://arxiv.org/abs/2507.13726}, 
}

@article {chan.gao.weak.approx,
    AUTHOR = {Chan, Wai Kiu and Gao, Haochen and Li, Han},
     TITLE = {Explicit result on equivalence of rational quadratic forms
              avoiding primes},
   JOURNAL = {J. Number Theory},
  FJOURNAL = {Journal of Number Theory},
    VOLUME = {225},
      YEAR = {2021},
     PAGES = {281--293},
      ISSN = {0022-314X,1096-1658},
   MRCLASS = {11E12},
  MRNUMBER = {4235262},
MRREVIEWER = {Stefan\ K\"uhnlein},
       DOI = {10.1016/j.jnt.2021.02.006},
       URL = {https://doi.org/10.1016/j.jnt.2021.02.006},
}

@misc{intermediate.data,
AUTHOR = {Brandhorst, Simon and Mezzedimi, Giacomo},
TITLE = {On hyperbolic lattices whose reduced automorphism group is virtually abelian, ancillary files},
HOWPUBLISHED = {\url{https://doi.org/10.5281/zenodo.16319157}},
}

@book {cassels.quadratic,
    AUTHOR = {Cassels, J. W. S.},
     TITLE = {Rational quadratic forms},
    SERIES = {London Mathematical Society Monographs},
    VOLUME = {13},
 PUBLISHER = {Academic Press, Inc. [Harcourt Brace Jovanovich, Publishers],
              London-New York},
      YEAR = {1978},
     PAGES = {xvi+413},
      ISBN = {0-12-163260-1},
   MRCLASS = {10C05 (10-01 15A63)},
  MRNUMBER = {522835},
MRREVIEWER = {Charles\ J.\ Parry},
}

@book {meara.quadratic,
    AUTHOR = {O'Meara, O. Timothy},
     TITLE = {Introduction to quadratic forms},
    SERIES = {Classics in Mathematics},
      NOTE = {Reprint of the 1973 edition},
 PUBLISHER = {Springer-Verlag, Berlin},
      YEAR = {2000},
     PAGES = {xiv+342},
      ISBN = {3-540-66564-1},
   MRCLASS = {11Exx},
  MRNUMBER = {1754311},
}

@book {conway.sloane.lattices,
    AUTHOR = {Conway, J. H. and Sloane, N. J. A.},
     TITLE = {Sphere packings, lattices and groups},
    SERIES = {Grundlehren der mathematischen Wissenschaften [Fundamental
              Principles of Mathematical Sciences]},
    VOLUME = {290},
   EDITION = {Third},
      NOTE = {With additional contributions by E. Bannai, R. E. Borcherds,
              J. Leech, S. P. Norton, A. M. Odlyzko, R. A. Parker, L. Queen
              and B. B. Venkov},
 PUBLISHER = {Springer-Verlag, New York},
      YEAR = {1999},
     PAGES = {lxxiv+703},
      %ISBN = {0-387-98585-9},
   MRCLASS = {11H31 (05B40 11H06 20D08 52C07 52C17 94B75 94C30)},
  MRNUMBER = {1662447},
MRREVIEWER = {Renaud Coulangeon},
       DOI = {10.1007/978-1-4757-6568-7},
       URL = {https://doi.org/10.1007/978-1-4757-6568-7},
}

@article {nikulin.finite.aut.greater.5,
    AUTHOR = {Nikulin, V. V.},
     TITLE = {Quotient-groups of groups of automorphisms of hyperbolic forms
              by subgroups generated by {$2$}-reflections},
     JOURNAL = {{A}lgebro-geometric applications},
 BOOKTITLE = {Current problems in mathematics, {V}ol. 18},
     PAGES = {3--114},
 PUBLISHER = {Akad. Nauk SSSR, Vsesoyuz. Inst. Nauchn. i Tekhn. Informatsii,
              Moscow},
      YEAR = {1981},
}

@article {nikulin.finite.aut.3,
    AUTHOR = {Nikulin, V. V.},
     TITLE = {{$K3$} surfaces with a finite group of automorphisms and a
              {P}icard group of rank three},
      NOTE = {Algebraic geometry and its applications},
   JOURNAL = {Trudy Mat. Inst. Steklov.},
  FJOURNAL = {Akademiya Nauk SSSR. Trudy Matematicheskogo Instituta imeni V.
              A. Steklova},
    VOLUME = {165},
      YEAR = {1984},
     PAGES = {119--142},
}

@article{vinberg.finite.aut.4,
 AUTHOR = {Vinberg, E. B.},
 TITLE = {Classification of 2-reflective hyperbolic lattices of rank 4},
 FJOURNAL = {Transactions of the Moscow Mathematical Society},
 JOURNAL = {Trans. Mosc. Math. Soc.},
 ISSN = {0077-1554},
 VOLUME = {2007},
 PAGES = {39--66},
 YEAR = {2007},
 DOI = {10.1090/S0077-1554-07-00160-4},
}

@article{nikulin.factor.groups,
 AUTHOR = {Nikulin, V. V.},
 TITLE = {Factor groups of groups of automorphisms of hyperbolic forms with respect to subgroups generated by 2-reflections. {Algebro}-geometric applications},
 FJOURNAL = {Journal of Soviet Mathematics},
 JOURNAL = {J. Sov. Math.},
 ISSN = {0090-4104},
 VOLUME = {22},
 PAGES = {1401--1475},
 YEAR = {1983},
 DOI = {10.1007/BF01094757},
}

@Article{galluzzi.lombardo.peters,
 AUTHOR = {Galluzzi, F. and Lombardo, G. and Peters, C.},
 TITLE = {Automorphs of indefinite binary quadratic forms and {{\(K3\)}}-surfaces with {Picard} number 2},
 FJOURNAL = {Rendiconti del Seminario Matematico. Universit{\'a} e Politecnico di Torino},
 JOURNAL = {Rend. Semin. Mat., Univ. Politec. Torino},
 ISSN = {0373-1243},
 VOLUME = {68},
 NUMBER = {1},
 PAGES = {57--77},
 YEAR = {2010},
}

@article{nikulin.interesting,
 AUTHOR = {Nikulin, V. V.},
 TITLE = {{{\(K3\)}} surfaces with interesting groups of automorphisms},
 FJOURNAL = {Journal of Mathematical Sciences (New York)},
 JOURNAL = {J. Math. Sci., New York},
 ISSN = {1072-3374},
 VOLUME = {95},
 NUMBER = {1},
 PAGES = {2028--2048},
 YEAR = {1999},
 DOI = {10.1007/BF02169159},
}

@unpublished {allcock.alternative,
    author = {Allcock, Daniel},
    title = {An Alternative to Vinberg's Algorithm},
    year = {2021},
    note = {\href{https://arxiv.org/abs/2110.03784}{arXiv:2110.03784}},
}

@article {brandhorst.veniani.hensel-lifting,
    AUTHOR = {Brandhorst, Simon and Veniani, Davide Cesare},
     TITLE = {Hensel lifting algorithms for quadratic forms},
   JOURNAL = {Math. Comp.},
  FJOURNAL = {Mathematics of Computation},
    VOLUME = {93},
      YEAR = {2024},
    NUMBER = {348},
     PAGES = {1963--1991},
      ISSN = {0025-5718,1088-6842},
   MRCLASS = {11E08 (11E12)},
  MRNUMBER = {4730253},
MRREVIEWER = {A.\ G.\ Earnest},
       DOI = {10.1090/mcom/3909},
       URL = {https://doi.org/10.1090/mcom/3909},
}

@article{shimada.connected.components,
 author = {Shimada, Ichiro},
 title = {Connected components of the moduli of elliptic {{\(K3\)}} surfaces},
 fjournal = {Michigan Mathematical Journal},
 journal = {Mich. Math. J.},
 issn = {0026-2285},
 volume = {67},
 number = {3},
 pages = {511--559},
 year = {2018},
 language = {English},
 doi = {10.1307/mmj/1528941621},
 zbMATH = {6969983},
 Zbl = {1420.14087}
}

@article {shimada.borcherds,
    AUTHOR = {Shimada, Ichiro},
     TITLE = {An algorithm to compute automorphism groups of {$K3$} surfaces
              and an application to singular {$K3$} surfaces},
   JOURNAL = {Int. Math. Res. Not. IMRN},
  FJOURNAL = {International Mathematics Research Notices. IMRN},
      YEAR = {2015},
    NUMBER = {22},
     PAGES = {11961--12014},
      ISSN = {1073-7928},
   MRCLASS = {14J28 (14J50)},
  MRNUMBER = {3456710},
MRREVIEWER = {Makiko Mase},
       DOI = {10.1093/imrn/rnv006},
       URL = {https://doi.org/10.1093/imrn/rnv006},
}

@article {plesken.souvignier,
    AUTHOR = {Plesken, W. and Souvignier, B.},
     TITLE = {Computing isometries of lattices},
      NOTE = {Computational algebra and number theory (London, 1993)},
   JOURNAL = {J. Symbolic Comput.},
  FJOURNAL = {Journal of Symbolic Computation},
    VOLUME = {24},
      YEAR = {1997},
    NUMBER = {3-4},
     PAGES = {327--334},
      ISSN = {0747-7171},
   MRCLASS = {11H56 (11Y16)},
  MRNUMBER = {1484483},
MRREVIEWER = {Christine Bachoc},
       DOI = {10.1006/jsco.1996.0130},
       URL = {https://doi.org/10.1006/jsco.1996.0130},
}

@article {yu.entropy,
    AUTHOR = {Yu, Xun},
     TITLE = {K3 surface entropy and automorphism groups},
   JOURNAL = {J. Algebraic Geom.},
  FJOURNAL = {Journal of Algebraic Geometry},
    VOLUME = {34},
      YEAR = {2025},
    NUMBER = {2},
     PAGES = {205--231},
      ISSN = {1056-3911,1534-7486},
   MRCLASS = {14J28 (37B40 37F80)},
  MRNUMBER = {4876289},
}

@misc{brandhorst.mezzedimi,
      title={Borcherds lattices and K3 surfaces of zero entropy}, 
      author={Simon Brandhorst and Giacomo Mezzedimi},
      year={2022, to appear in \textit{Amer. J. Math.}},
      eprint={2211.09600},
      archivePrefix={arXiv},
      url={https://arxiv.org/abs/2211.09600}, 
}

@article {tits,
    AUTHOR = {Tits, J.},
     TITLE = {Free subgroups in linear groups},
   JOURNAL = {J. Algebra},
  FJOURNAL = {Journal of Algebra},
    VOLUME = {20},
      YEAR = {1972},
     PAGES = {250--270},
      ISSN = {0021-8693},
   MRCLASS = {20.75},
  MRNUMBER = {286898},
MRREVIEWER = {B.\ A. F. Wehrfritz},
       DOI = {10.1016/0021-8693(72)90058-0},
       URL = {https://doi.org/10.1016/0021-8693(72)90058-0},
}

\end{document}